\documentclass[10pt,twoside,a4paper]{article}

\usepackage{graphicx}

\usepackage[utf8]{inputenc}

\usepackage[T1]{fontenc}
\usepackage{lmodern}
\usepackage{bm}
\usepackage{amsmath}
\numberwithin{equation}{section}

\usepackage{amssymb}

\usepackage{microtype}

\usepackage[boxed]{algorithm2e}

\usepackage{subfigure}

\usepackage[backend=biber,maxnames=10,backref=true,hyperref=true,giveninits=true,safeinputenc]{biblatex}
\DefineBibliographyStrings{english}{%
	backrefpage = {cited on page},
	backrefpages = {cited on pages},
}

\DeclareFieldFormat[report]{title}{``#1''}
\DeclareFieldFormat[book]{title}{``#1''}
\AtEveryBibitem{\clearfield{url}}
\AtEveryBibitem{\clearfield{note}}

\usepackage{graphicx}
\graphicspath{{images/}}
\usepackage{epstopdf}

\usepackage{tikz}
\usepackage{pgfplots}
\pgfplotsset{compat=newest}

\title{Data driven reconstruction using frames and Riesz bases}
\author{Andrea Aspri$^{1}$\\{\footnotesize\href{mailto:andrea.aspri@unipv.it}{andrea.aspri@unipv.it}}
\and Leon Frischauf$^{2}$\\{\footnotesize\href{mailto:leon.frischauf@univie.ac.at}{leon.frischauf@univie.ac.at}}
\and Yury Korolev$^{3}$\\{\footnotesize\href{mailto:y.korolev@maths.cam.ac.uk}{y.korolev@maths.cam.ac.uk}}
\and Otmar Scherzer$^{2,4}$\\{\footnotesize\href{mailto:otmar.scherzer@univie.ac.at}{otmar.scherzer@univie.ac.at}}}

\date{\today}

\usepackage[pdftex,colorlinks=true,linkcolor=blue,citecolor=green,urlcolor=blue,bookmarks=true,bookmarksnumbered=true]{hyperref}
\hypersetup
{
    pdfauthor={Aspri et al.},
    pdfsubject={Subject},
    pdftitle={Data driven reconstruction using frames and Riesz bases},
    pdfkeywords={Keywords}
}

\usepackage[hyperref,amsmath,thmmarks]{ntheorem}
\usepackage{aliascnt}

\newtheorem{lemma}{Lemma}[section]

\newaliascnt{proposition}{lemma}
\newtheorem{proposition}[proposition]{Proposition}
\aliascntresetthe{proposition}

\newaliascnt{corollary}{lemma}

\aliascntresetthe{corollary}

\newaliascnt{theorem}{lemma}
\newtheorem{theorem}[theorem]{Theorem}
\aliascntresetthe{theorem}

\theorembodyfont{\normalfont}
\newaliascnt{definition}{lemma}
\newtheorem{definition}[definition]{Definition}
\aliascntresetthe{definition}

\newaliascnt{assumption}{lemma}
\newtheorem{assumption}[assumption]{Assumption}
\aliascntresetthe{assumption}

\newaliascnt{notation}{lemma}

\aliascntresetthe{notation}

\newaliascnt{example}{lemma}

\aliascntresetthe{example}

\newaliascnt{remark}{lemma}
\newtheorem{remark}[remark]{Remark}
\aliascntresetthe{remark}

\theoremstyle{nonumberplain}
\theoremseparator{:}
\theoremheaderfont{\normalfont\itshape}

\theoremsymbol{\ensuremath{\square}}
\newtheorem{proof}{Proof}

\usepackage[a4paper,centering,bindingoffset=0cm,marginpar=2cm,margin=2.5cm]{geometry}

\usepackage[pagestyles]{titlesec}
\titleformat{\section}[block]{\large\sc\filcenter}{\thesection.}{0.5ex}{}[]
\titleformat{\subsection}[runin]{\bf}{\thesubsection.}{0.5ex}{}[.]

\newpagestyle{headers}%
{%
	\headrule
	\sethead%
		[\footnotesize\thepage]%
		[\footnotesize\sc Data driven reconstruction using frames and Riesz bases]%
		[]%
		{}%
		{\footnotesize\sc Data driven reconstruction using frames and Riesz bases}%
		{\footnotesize\thepage}
	\setfoot{}{}{}
}
\usepackage[font=footnotesize,format=plain,labelfont=sc,textfont=sl,width=0.75\textwidth,labelsep=period]{caption}

\usepackage[boxed]{algorithm2e}

\usepackage{geometry}
\usepackage{float}

\usepackage{dsfont}
\newcommand{\N}{\mathds{N}}

\let\RE\Re
\let\Re=\undefined
\DeclareMathOperator{\Re}{\RE e}
\let\IM\Im
\let\Im=\undefined
\DeclareMathOperator{\Im}{\IM m}

\newcommand{\abs}[1]{\left|#1\right|}
\newcommand{\norm}[1]{\left\|#1\right\|}
\newcommand{\set}[1]{\left\{#1\right\}}
\newcommand{\inner}[2]{\left<#1,#2\right>}

\let\ii\i
\renewcommand{\i}{\mathrm i}

\usepackage{enumerate}

\renewcommand{\leq}{\leqslant}
\renewcommand{\geq}{\geqslant}
\renewcommand{\phi}{\varphi}

\let\span\relax
\DeclareMathOperator{\span}{Span}

\newcommand{\U}{\mathcal U}
\newcommand{\Y}{\mathcal Y}

\newcommand{\udag}{u^{\dagger}}

\begin{document}
\maketitle
\thispagestyle{empty}
\parbox[t]{14em}{\footnotesize
\hspace*{-1ex}$^1$ Department of Mathematics\\
University of Pavia\\
Via Ferrata, 5 - 27100 Pavia
}
\parbox[t]{12em}{\footnotesize
	\hspace*{-1ex}$^2$Faculty of Mathematics\\
	University of Vienna\\
	Oskar-Morgenstern-Platz 1\\
	A-1090 Vienna, Austria}
\parbox[t]{14em}{\footnotesize
	\hspace*{-1ex}$^3$Department of Applied Mathematics and Theoretical Physics\\
	University of Cambridge\\
	Wilberforce Road, Cambridge CB3 0WA, United Kingdom}\\
\parbox[t]{14em}{\footnotesize
\hspace*{-1ex}$^4$
Johann Radon Institute \\
for Computational
and Applied Mathematics (RICAM)\\
Altenbergerstraße 69\\
A-4040 Linz, Austria}

\begin{abstract}
We study the problem of regularization of inverse problems adopting a purely data driven approach,
by using the similarity to the method of regularization by projection. 
We provide an application of a projection algorithm, utilized and applied in frames theory, as a data driven reconstruction procedure in inverse problems, generalizing the algorithm proposed by the authors in \textit{Inverse Problems} \textbf{36} (2020), n. 12, 125009, based on an orthonormalization procedure for the training pairs. We show some numerical experiments, comparing the different methods.
\end{abstract}

\textit{Keywords and phrases}. {Data driven reconstructions, Gram-Schmidt procedure, Inverse problems, Frames, Riesz bases}
%

\textit{2010 Mathematics Subject Classification}. {65J22, 47A52, 42C15}
 
\ \\ \ \\
The authors dedicate this paper to Zuhair Nashed. Otmar Scherzer is grateful for Zuhair's long lasting mentorship, his friendship, and the personal and professional exchange with him.

\section{Introduction}\label{sec:intro}
Inverse problems are concerned with the reconstruction of an unknown quantity $u \in \U$ from its indirect measurements $y \in \Y$ which are related by the \emph{forward model} $T \colon \U \to \Y$ that describes the relationship between the quantities of interest and the measurements. 
The forward operator $T$ models the physics of data acquisition and may involve, for instance, integral transforms (such as the Radon transform, see for instance \cite{Nat01,NatWue01,Kuc13}) and partial differential equations (PDEs) (see for instance~\cite{Isa90,Sym09}). 

Previously inverse problems were considered \emph{model driven}, meaning that the physics and chemistry of the acquisition process were expressed as accurately as possible with mathematical formulas. Typically this allows to write the inverse problem as an operator equation
\begin{equation}\label{eq:Au=y}
Tu=y\;.
\end{equation}
With the rise of the area of big data, \emph{data driven} approaches (that means \emph{avoiding} modeling of the operator $T$) have emerged (see \cite{ArrMaaOktScho19}). They are attractive because they do not require the forward operator explicitly, and often yield superior visual quality of the reconstruction. Instead these approaches require a series of training data.

Currently there is no adequate theory for purely data driven regularisation 
in inverse problems, i.e. a theory in the setting when the forward operator is given only via $n$ \emph{training pairs} 
\begin{equation}\label{eq:pairs}
\set{u_i, y_i } \quad \text{such that} \quad T u_i= y_i \quad \text{for all} \quad i=1,\ldots,n.
\end{equation}
We call $\{u_i\}$ \emph{training inputs} and $\{y_i\}$ \emph{training outputs}, respectively.
In \cite{AspKorSch20} we made a first step of an analysis for purely data driven regularization by utilizing the similarity to the concept of \emph{regularization by projection}.
We demonstrated that regularisation by projection~\cite{Sei80,EngHanNeu96} and variational regularisation~\cite{SchGraGroHalLen09} can be formulated in a data driven setting and usual results such as convergence and stability can be obtained. 

As mentioned above, the proposed method in \cite{AspKorSch20} does not require the 
explicit knowledge of the operator $T$ but only information on training pairs. For practical applications and in theoretical considerations 
the training pairs need to be orthonormalized, which we implemented via Gram-Schmidt in \cite{AspKorSch20}.
This paper evaluates alternatives to the Gram-Schmidt orthonormalization, such as Householder reflections, QR decomposition, and in particular \emph{frame decompositions}. The latter can be  
analyzed along the lines of \cite{AspKorSch20} when the frame forms a \emph{Riesz basis}. 

The paper is organized as follows. In \autoref{sec:Gram_Schmidt} we review basic results from \cite{AspKorSch20}. In \autoref{sec:riesz} we review basic facts on \emph{frames} and \emph{Riesz bases}, which are applied in \autoref{sec:data_driven} to data driven regularization. Finally we present 
numerical results in \autoref{sec:simulations}.

\section{Gram-Schmidt orthonormalization as in \cite{AspKorSch20}}\label{sec:Gram_Schmidt}
In this section, we recall the main idea used in \cite{AspKorSch20} of data-driven projection methods for reconstructions. 

Let $T$ be an operator acting between Hilbert spaces, i.e., $T:\mathcal{U}\to \mathcal{Y}$. The operator $T$ is assumed to be linear, bounded and injective. 
We consider \autoref{eq:Au=y} from the introduction.

Let $\mathcal{U}_n$ and $\mathcal{Y}_n$ be finite-dimensional subspaces (of dimension $n$) of the Hilbert spaces $\mathcal{U}$ and $\mathcal{Y}$, respectively. Let $P_{H_n}$ represent the orthogonal projection operator onto $H_n$, which is either $H_n=\mathcal{U}_n$ or $H_n=\mathcal{Y}_n$.

In the sequel we denote by $\udag$ the solution of \autoref{eq:Au=y} and we assume that the 
\textit{training inputs} $\set{u_i}$ from \autoref{eq:pairs} are
linearly independent and consequently, due to the injectivity of the operator $T$, the same holds for $\set{y_i}$. 

Finally, it is assumed that $\mathcal{U}_n\subset \mathcal{U}_{n+1}$ and $\mathcal{Y}_n\subset \mathcal{Y}_{n+1}$, for all $n$ and
    \begin{equation}
        \overline{\bigcup_{n\in\N}\mathcal{U}_n}=\mathcal{U},\qquad \overline{\bigcup_{n\in\N}\mathcal{Y}_n}=\overline{{R(T)}}.
    \end{equation}

\textit{Regularization by projection} consists in approximating the solution $\udag$ of \autoref{eq:Au=y} by the \textit{minimum norm solution} of the projected equation
    \begin{equation}\label{eq: TPnu=y}
        T P_{\mathcal{U}_n}u=y,
    \end{equation}
where $P_{\mathcal{U}_n}$ is the orthogonal projection operator onto $\mathcal{U}_n$.    
The minimum norm solution of this equation is unique and is given by
    \begin{equation} \label{eq:udn2}
        u_n^{\mathcal{U}}=(T P_{{\mathcal{U}_n}})^{\dagger} y,
    \end{equation}
where $(T P_{{\mathcal{U}_n}})^{\dagger}$ denotes the Moore-Penrose inverse of the operator $T P_{{\mathcal{U}_n}}$ (\cite{Nas76}). The projection takes place in the space $\mathcal{U}$, hence the superscript in our notation $u_n^{\mathcal{U}}$.
\begin{center}
    The topic of \cite{AspKorSch20} was to find $u_n^{\mathcal{U}}$ by making use of the training pairs \autoref{eq:pairs} without explicit knowledge of the operator $T$.
\end{center}
This was numerically realized by application of the Gram-Schmidt orthonormalization procedure to the training outputs $\set{y_i}_{i=1}^{\infty}$, resulting in an orthonormal basis $\set{\underline{y_i}}_{i=1}^{\infty}$. That is     
    \begin{equation*}
    	\underline{y_i} = \frac{y_i-\sum_{k=1}^{i-1} \inner{y_i}{\underline{y_k}} \underline{y_k}}
    	                       {\|y_i-P_{\mathcal{Y}_{i-1}}y_i\|}\quad \text{for all} \quad i\in \N,
    \end{equation*}
%
        and consequently with $T \underline{u_i} = \underline{y_i}$ we get
        \begin{equation*}
            \underline{u_i} = \frac{u_i-\sum_{k=1}^{i-1}\inner{ y_i}{\underline{y_k}} \underline{u_k}}{\|y_i-P_{\mathcal{Y}_{i-1}}y_i\|}\quad \text{for all} \quad i\in \N;
        \end{equation*}
As shown in \cite{AspKorSch20}, $(T P_{{\mathcal{U}_n}})^{\dagger}=T^{-1}P_{\mathcal{Y}_n}$, hence we get the following  reconstruction formula 
        \begin{equation}\label{eq:udn2}
        \boxed{
            u_n^{\mathcal{U}}=T^{-1}P_{\mathcal{Y}_n}y
            =\sum_{i=1}^{n}\inner{y}{\underline{y_i}} \underline{u_i}
            \quad \text{ and } \quad 
            Tu_n^{\mathcal{U}}
            =\sum_{i=1}^{n}\inner{y}{\underline{y_i}} \underline{y_i}.}
        \end{equation}

\begin{remark}
We stress that this algorithm doesn't require the explicit knowledge of the operator $T$.
\end{remark}
\subsection{Weak convergence}
We recall a weak convergence result from \cite{AspKorSch20}, which is actually formulated for the orthonormalized training inputs $\set{\overline{u_i}}$ of $\set{u_i}$ via Gram-Schmidt. Moreover, we define $T\overline{u_i}=\overline{y_i}$. 

To prove weak convergence of the reconstruction formula \autoref{eq:udn2} for $n \to \infty$, 
we posed in \cite{AspKorSch20} some assumptions on  $\set{\overline{u_i},\overline{y_i}}$. This seems better suited for inverse problems, because if 
$\set{\overline{u_i}}$ is orthonormal, the sequence $\set{\norm{\overline{y_i}}}$ can be expected to converge to $0$, or in other words one may expect some decay in the coefficients of the expansions of $\set{\overline{y_i}}$.
\begin{assumption}\label{ass: prev_work}
Let
\begin{enumerate}
    \item $\sum_{i=1}^\infty|\langle \udag, \overline{u_i}\rangle|<\infty$;
    \item For every $n\in\N$ and any $i\geq n+1$ consider the following expansion $P_{\mathcal{Y}_n}\overline{y_i}=\sum_{j=1}^{n}\beta^{i,n}_j \overline{y_j}$. We assume that there exists some $C>0$ such that for every $n\in\N$ and every $i\geq n+1$, $\sum_{j=1}^{n}(\beta^{i,n}_j)^2\leq C$.
\end{enumerate}
\end{assumption}
 
\begin{theorem}[Theorems 9 and 11 in \cite{AspKorSch20}] \label{th:weak}
Let $y$ be the exact right-hand side of \autoref{eq:Au=y} and $\set{u_i,y_i}_{i=1}^{n}$ the training pairs defined in \autoref{eq:pairs}. Let \autoref{ass: prev_work} hold, then $u_n^{\mathcal{U}}$ converges weakly to $\udag$.
\end{theorem}

\section{Basics on Frames and Riesz-bases} \label{sec:riesz}
This section is devoted to collect some notations, utilized in the rest of the paper, and useful results on frames' theory.

In this section, $H$ represents a generic Hilbert space.
We denote by $\norm{\cdot}$ the norm induced by the inner product in $H$, denoted by $\inner{\cdot}{\cdot}$. 

\begin{definition}
	A sequence $\set{f_i}_{i=1}^\infty$ of elements in $H$ is a frame for $H$ if there exist constants $A,B>0$ such that 
		\begin{equation}\label{def:frames}
			A\|f\|^2\leq \sum_{i=1}^\infty|\inner{f}{f_i}|^2\leq B \|f\|^2\quad  \text{for all} \quad  f\in H,
		\end{equation}
where $A,B$ are called frames bounds. 
\end{definition}
It follows from the definition that if $\set{f_i}_{i=1}^\infty$ is a frame for $H$, then 
\begin{equation*}
	\overline{\textrm{span}\set{f_i}}_{i=1}^\infty=H\:.
\end{equation*}
There exist some operators associated to a frame
\begin{itemize}
	\item $F:l^2(\N)\to H$, called \textit{synthesis operator}
	\begin{equation*}
		F c_i=\sum_{i=1}^\infty c_i f_i.
	\end{equation*}
	\item $F^*:H\to l^2(\N)$, the adjoint operator of $F$, called \textit{analysis operator}
	\begin{equation*}
		F^*f=\set{\inner{f}{f_i} }_{i=1}^\infty.
	\end{equation*}
	\item $S:H\to H$, called \textit{frame operator}
	\begin{equation*}
		S= FF^*, \quad Sf=FF^*f=\sum_{i=1}^\infty\inner{f}{f_i} f_i.
	\end{equation*} 
    $S$ is bounded, invertible, self-adjoint and positive and for every $f \in H$
    \begin{equation*}
    	Sf=\sum_{i=1}^\infty \inner{f}{f_i} f_i
    \end{equation*}
    is unconditionally convergent (see \cite{Chr16}).
\end{itemize}
The following theorem follows from the properties of $S$.
\begin{theorem} \label{th:frame}
	Let $\set{f_i}_{i=1}^\infty$ be a frame with frame operator $S$. Then 	
	\begin{equation}\label{eq: rec1}
		f=\sum_{i=1}^\infty\langle f, S^{-1}f_i\rangle f_i \quad \text{for all} \quad  f \in H
	\end{equation}
	and
	\begin{equation}\label{eq: rec2}
		f=\sum_{i=1}^\infty\inner{f}{f_i} S^{-1} f_i \quad \text{for all} \quad  f \in H .
	\end{equation}
	Both series are unconditionally convergent.
\end{theorem} 
For our purposes we are interested in a special class of frames, that is  Riesz bases. For reader's convenience, we recall here their definition, see for more details \cite{Chr16}.
\begin{definition}[Riesz's basis]
	Let $\set{e_i}_{i=1}^\infty$ be an orthonormal basis for $H$. A Riesz basis $\set{f_i}_{i=1}^\infty$ for $H$ is a family of the form $\set{f_i}_{i=1}^\infty=\set{L e_i}_{i=1}^\infty$ where $L:H\to H$ is a bounded and bijective operator. 
\end{definition}
As a consequence of the previous definition, a Riesz basis is $\omega$-independent, that is 
	\begin{equation*}
		\sum_{i=1}^\infty c_i f_i=0\quad \Rightarrow\quad c_i=0\quad  \textrm{for all}\ \ i.
	\end{equation*}
We summarize some of the properties of Riesz basis in the following proposition.
\begin{proposition}\label{prop: Riesz basis}
	A Riesz basis $\set{f_i}_{i=1}^\infty$ for $H$ is a frame for $H$, i.e., it satisfies \autoref{def:frames}, and the Riesz basis bounds coincide with the frame bounds $A$ and $B$. Moreover 
	\begin{enumerate}
	\item $\set{f_i}_{i=1}^\infty$ and $\set{S^{-1}f_i}_{i=1}^\infty$ are biorthogonal, i.e., $\langle f_i, S^{-1}f_j\rangle =\delta_{ij}$, where $\delta_{ij}$ is the Kronecker symbol; 
	\item for each $f\in H$ there exists a unique sequence of scalars $\set{c_i}_{i=1}^\infty$ such that $f=\sum_{i=1}^\infty c_i f_i$ and $\sum_{i=1}^\infty|c_i|^2<\infty$;
	\item for every finite scalar sequence $\set{c_i}$, it holds
	\begin{equation}\label{eq:finite_sequence}
		A\sum_{i=1}^\infty |c_i|^2\leq \Big\|\sum_{i=1}^\infty c_if_i\Big\|^2\leq B\sum_{i=1}^\infty |c_i|^2.
	\end{equation} 
	\end{enumerate}
\end{proposition}
Finally, as a consequence of the previous properties for Riesz basis, it holds
\begin{proposition} \label{pr:Riesz}
Let $J$ be a countable index set. Any subfamily $\set{f_i}_{i\in J}$ is a Riesz basis for its closed linear spanning set $\overline{\textrm{span}\set{f_i}}_{i\in J}$, with bounds $A$ and $B$.  	
\end{proposition}
We refer to \cite{Chr96,Chr16} for more details and some literature on the topic.

\section{Data driven regularization by frames and Riesz bases}\label{sec:data_driven}
We propose a reconstruction algorithm based on projection methods onto finite-dimensional subspaces, similar to the one discussed in \autoref{sec:Gram_Schmidt}. 
However, now, compared with Section \ref{sec:Gram_Schmidt}, we consider the case when $\set{y_i}_{i=1}^\infty$ forms 
a frame for $\mathcal{Y}$. Associated with the frame is the synthesis operator $F$ and the Frame operator $S$ on $\mathcal{Y}$. We assume that $\set{y_i}_{i=1}^{n}$ is a frame on 
$\mathcal{Y}_n:=\span \set{y_i: i=1,\ldots,n}$ as well. 
The corresponding \emph{restricted frame operator} $S_n:\mathcal{Y}_n \to \mathcal{Y}_n$ is given by $S_n y = \sum_{i=1}^{n}\langle y, y_i\rangle y_i$, and therefore because of \autoref{th:frame}
\begin{equation}\label{eq: yn}
	y=\sum_{i=1}^{n}\inner{y}{S_n^{-1}y_i} y_i \quad \text{ for all } y \in \mathcal{Y}_n.
\end{equation}
Note that for every $y \in \mathcal{Y}$, $P_{\mathcal{Y}_n}y \in \mathcal{Y}_n$ and therefore, because $\inner{P_{\mathcal{Y}_n}y}{y_j}=\inner{y}{y_j}$ we get
\begin{equation}\label{eq: projection_formula}
	P_{\mathcal{Y}_n}y=\sum_{i=1}^{n}\inner{P_{\mathcal{Y}_n}y}{S^{-1}_n y_i} y_i=\sum_{i=1}^{n}\inner{y}{S^{-1}_n y_i} y_i \quad \text{ for all } y \in \mathcal{Y}. 
\end{equation}
Using the injectivity hypothesis on the operator $T$, we find that
\begin{equation}\label{eq:reconstruction formula}
\boxed{
	u_n^{\mathcal{U}}=\sum_{i=1}^{n}\inner{y}{S_n^{-1}y_i} u_i \text{ and }
Tu_n^{\mathcal{U}}=\sum_{i=1}^{n}\inner{y}{S_n^{-1}y_i} y_i.}
\end{equation}
\begin{remark}
This is the comparable formula for the Gram-Schmidt orthonormalization procedure \autoref{eq:udn2}.
In order to identify the unknown coefficients $\inner{y}{S^{-1}_n y_i}$, we proceed as in \cite{AdcHuy19}: It follows from \autoref{eq: projection_formula} that
\begin{equation}\label{eq: coeff_expansion} 
    \inner{y}{y_j} = \inner{P_{\mathcal{Y}_n}y}{y_j} = 
	\inner{ \sum_{i=1}^{n}\inner{y}{S^{-1}_n y_i}y_i}{y_j} = \sum_{i=1}^{n}\inner{y}{S_n^{-1}y_i} \inner{y_i}{y_j}\quad \text{for all} \quad y_j\in \mathcal{Y}_n,
\end{equation}
or in other words 
\begin{equation}\label{eq:gram_matrix_eq}
	G_{\mathcal{U}_n} X=Y \text{ with } G_{\mathcal{U}_n}=\left(\inner{y_i}{y_j} \right)_{i,j=1}^{n}, \;
    X=\left(\inner{y}{S^{-1}_n y_i} \right)_{i=1}^{n} \text{ and } Y=\left(\inner{y}{y_j} \right)_{j=1}^{n}.
\end{equation}
\autoref{eq: coeff_expansion} can be implemented for reconstruction.
\end{remark}

\subsection{Weak convergence}
Following the analysis in \cite{AspKorSch20} a similar result as \autoref{th:weak} on weak convergence can be obtained for Riesz bases. In this case Gram-Schmidt orthormalization is 
replaced by the calculation of dual frame.  
\begin{assumption}
Now, we assume that the training inputs 
	\begin{equation}\label{ass:ui_Riesz_basis} 
		\set{u_i}^\infty_{i=1}\ \ \textrm{form a Riesz basis for}\ \ \mathcal{U}.
	\end{equation}
\end{assumption}
We emphasize that the assumption that the image data form a Riesz basis is only necessary 
for the theoretical analysis. In practical applications we only require the knowledge of a dual frame to implement \autoref{eq:reconstruction formula}.

As a consequence of \autoref{eq:  rec2} we see that  
\begin{equation*}
	\udag=\sum_{i=1}^\infty \inner{ \udag}{S^{-1} u_i} u_i \quad \text{ and therefore also } \quad
	y=\sum_{i=1}^\infty\inner{\udag}{S^{-1} u_i} y_i.
\end{equation*}
Then, we consider the projection onto $\mathcal{Y}_n$ that is
\begin{equation*}
	P_{\mathcal{Y}_n}y=\sum_{i=1}^\infty\inner{\udag}{S^{-1} u_i} P_{\mathcal{Y}_n}y_i
\end{equation*}
Note that in comparison with \autoref{eq: projection_formula} here the inverse of the frame operator $S$ on $\mathcal{U}$ is used, and not $S_n$ on $\mathcal{U}_n$. 

We investigate weak convergence of the \emph{Riesz bases based approximation}: 
\begin{definition}[Riesz bases based approximation]
We define 
\begin{equation}\label{eq:utilden}
	\widetilde{u}_n^\mathcal{U}:=T^{-1} P_{\mathcal{Y}_n}y = \sum_{i=1}^\infty\inner{\udag}{S^{-1} u_i} T^{-1}P_{\mathcal{Y}_n}y_i.
\end{equation}
\end{definition}
We can represent ${T^{-1}P_{\mathcal{Y}_n}y_i}$ in terms of the Riesz basis, for all $n$ and $i$, that is
	\begin{equation}\label{eq:expansion_of_Tm1_Py_n}
		{T^{-1}P_{\mathcal{Y}_n}y_i}=\sum_{j=1}^\infty c^{i,n}_j u_j. 
	\end{equation}
To study the convergence of the sequence $\widetilde{u}_n^\mathcal{U}$, we need some assumptions about the coefficients of the expansion in \autoref{eq:utilden} and \autoref{eq:expansion_of_Tm1_Py_n}.
\begin{lemma}\label{lemma:coeff}
Let $\set{u_i}_{i=1}^\infty$ be a Riesz basis. If $(\inner{\udag}{S^{-1} u_i})_{i=1}^\infty \in l^1$ and, for every $i$ and $n$, $\sum_{j=1}^\infty \abs{c^{i,n}_j}^2< C$, where $C$ is independent of $i$ and $n$, then  $\|\widetilde{u}_n^\mathcal{U}\|$ is bounded.
\end{lemma}
\begin{proof}
Apply H\"{o}lder inequality to \autoref{eq:utilden}, i.e.,  
\begin{equation*}
	\|\widetilde{u}_n^\mathcal{U}\|\leq \sum_{i=1}^\infty|\inner{\udag}{S^{-1} u_i}|\ \ \textrm{sup}_{i}\|T^{-1}P_{\mathcal{Y}_n}y_i\|
\end{equation*}
and, thanks to \autoref{eq:expansion_of_Tm1_Py_n} and the assumption on $\sum_{j=1}^\infty \abs{c^{i,n}_j}^2$, we have that 
\begin{equation*}
	\norm{T^{-1}P_{\mathcal{Y}_n}y_i}^2= \norm{ \sum_{j=1}^\infty c^{i,n}_j u_j }^2\leq B \sum_{j=1}^\infty \abs{c^{i,n}_j}^2<\infty.
\end{equation*}
The assertion follows using the hypothesis on the coefficients $(\inner{\udag}{S^{-1} u_i})_{i=1}^\infty$. 
\end{proof}
\begin{theorem}
 Let the hypothesis of \autoref{lemma:coeff} hold. Then the sequence \autoref{eq:utilden} is weakly convergent to $\udag$.
\end{theorem}

In the next section, we show some numerical experiments and comparisons between orthonormalization procedures and the reconstruction formula \autoref{eq:reconstruction formula}. 

\section{Numerical Experiments}\label{sec:simulations}

The goal of this section is to present numerical experiments illustrating the reconstruction with \autoref{eq:reconstruction formula} and \autoref{eq: coeff_expansion} and compare it with different orthonormalization procedures, such as Gram-Schmidt originally proposed in \cite{AspKorSch20}, Householder reflections and the QR decomposition.

\textbf{General Structure of the Experiments.} For the numerical experiments, the operator $T:L^2(\Omega)\longrightarrow L^2(\mathbb R \times [0,\pi))$ is the Radon transform with a parallel beam geometry (see \cite{Kuc13}), which is the same example as considered in \cite{AspKorSch20}. The finite-dimensional training pairs are denoted in this section via the notation $\set{\vec{u}_i,\vec{y}_i}_{i=1}^n$, where $\vec{u}_i\in \mathbb{R}^k$ and $\vec{y}_i\in\mathbb{R}^h$, for some $h,k\in\mathbb{N}$. 

Each data set includes $n$ different pictures, where each picture consists of $N$ pixels represented as $N$ points, with a gray scale intensity in $[0,1]$, which can be represented as linearly independent elements of $\mathbb R^{N}$. They are used as training data and Radon transformed by the built-in MATLAB-function for $K$ different angles with $\theta_k\in [0,\pi)$. This provides $n$ different elements of $\mathbb R^{M\times K} \simeq \mathbb R^{M\cdot K}$, where $M$ is the length of a Radon projection at a specific angle. In the case of quadratic pictures with size $\sqrt{N}\times\sqrt{N}$, this implies that $M\approx \sqrt{2\cdot N},$ which is the length of the diagonal of the square. In order to receive the preimages, the orthonormal system $\{\vec{\underline{y_1}} , \ldots , \vec{\underline{y_n}}\}$ is backtransformed via the exact inverse Radon transform. This 	makes it possible to investigate the quality of the procedure by comparing these results to the validation data sets.

\subsection{Orthonormalization procedures}\label{sec:simulations_Ortho}
In this section we compare \emph{Gram-Schmidt orthonormalization procedure}, \emph{QR decomposition} and \emph{Householder reflection}
for solving \autoref{eq:udn2}.

\textbf{Gram-Schmidt method.}  \label{sec:gsm}
    The vectors defined by $c_i := \set{\vec{y}_i - \sum_{j=1}^{i-1} \langle \vec{y}_i,\vec{y}_j \rangle \vec{y}_j}_{i=1}^n$
    are an orthogonal system (see \cite{Golub96}). After normalizing, we get the orthonormal system by
    $\vec{\underline{y_i}}  := \frac{c_i}{\|c_i\|}.$  The computational effort of this algorithm is $\mathcal O(MKn^2)$. It is possible to rearrange these calculation steps to make the computational procedure more stable of numerical errors, whereas the computational effort stays the same. This algorithm can be found in \cite{Golub96} and is used here for further computations.

\textbf{QR decomposition.} \label{sec:qr}
    By writing all vectors $\vec{y}_i$ as columns of a matrix and performing the MATLAB native $QR$ decomposition, we receive an orthonormal system with the same span as the columns of the original matrix, since the columns of $Q$ represent an orthonormal system. The computational effort of this algorithm is $\mathcal O(M^2K^2 n - M K n^2 + n^3/3)$.

\textbf{Householder reflection method.}\label{sec:householder}
    By applying Householder transformations to the set, one receives an orthonormal system of vectors. The computational effort of this algorithm is $\mathcal O(M^2K^2 n - M K n^2 + n^3/3)$.

\subsubsection{Accuracy of the orthonormalization procedures}

 We investigate the stability of the algorithms on the example of our imaging application. We assume the set $\set{\vec{\underline{y_i}}}_{i=1}^n$ being available and analyse the error of the resulting orthonormal system. We first define an appropriate measure of this error, which quantifies the success of the orthonormalization procedure as a numerical value.
\begin{definition}
For a set $\set{\vec{\underline{y_i}}}_{i=1}^n \subset \mathbb R^{M\cdot K}$, which is assumed to be approximately orthonormal, we define the matrix $\underline{Y} \in \mathbb R^{n\times n }$ by
\begin{equation*}
    (\underline{Y_{i j}}) := \langle \vec{\underline{y_i}}, \vec{\underline{y_j}} \rangle.
\end{equation*}
Furthermore, we define the orthonormality error $\epsilon_\text{ortho}$ by
\begin{equation*}
\epsilon_\text{ortho} := \| \underline{Y} - I \|,
\end{equation*}
where $I$ denotes the $\mathbb R^{n\times n}$ identity matrix and $\|\cdot\|$ maximum absolute column sum of the matrix.
\end{definition}

We test the methods on the example of the Sunflower data set from \href{www.kaggle.com/alxmamaev/flowers-recognition}{www.kaggle.com/alxmamaev/flowers-recognition} with different numbers of images and plot the error $\epsilon_\text{ortho}$ over the number $n$. We furthermore investigate the impact of random permutation of the images, before the method is applied. The results can be observed in Fig.~\ref{fig:compar1}-Fig.~\ref{fig:compar3}.


\begin{figure}[H]
  \centering
  \begin{minipage}[b]{0.3\textwidth}
    \includegraphics[width=\textwidth]{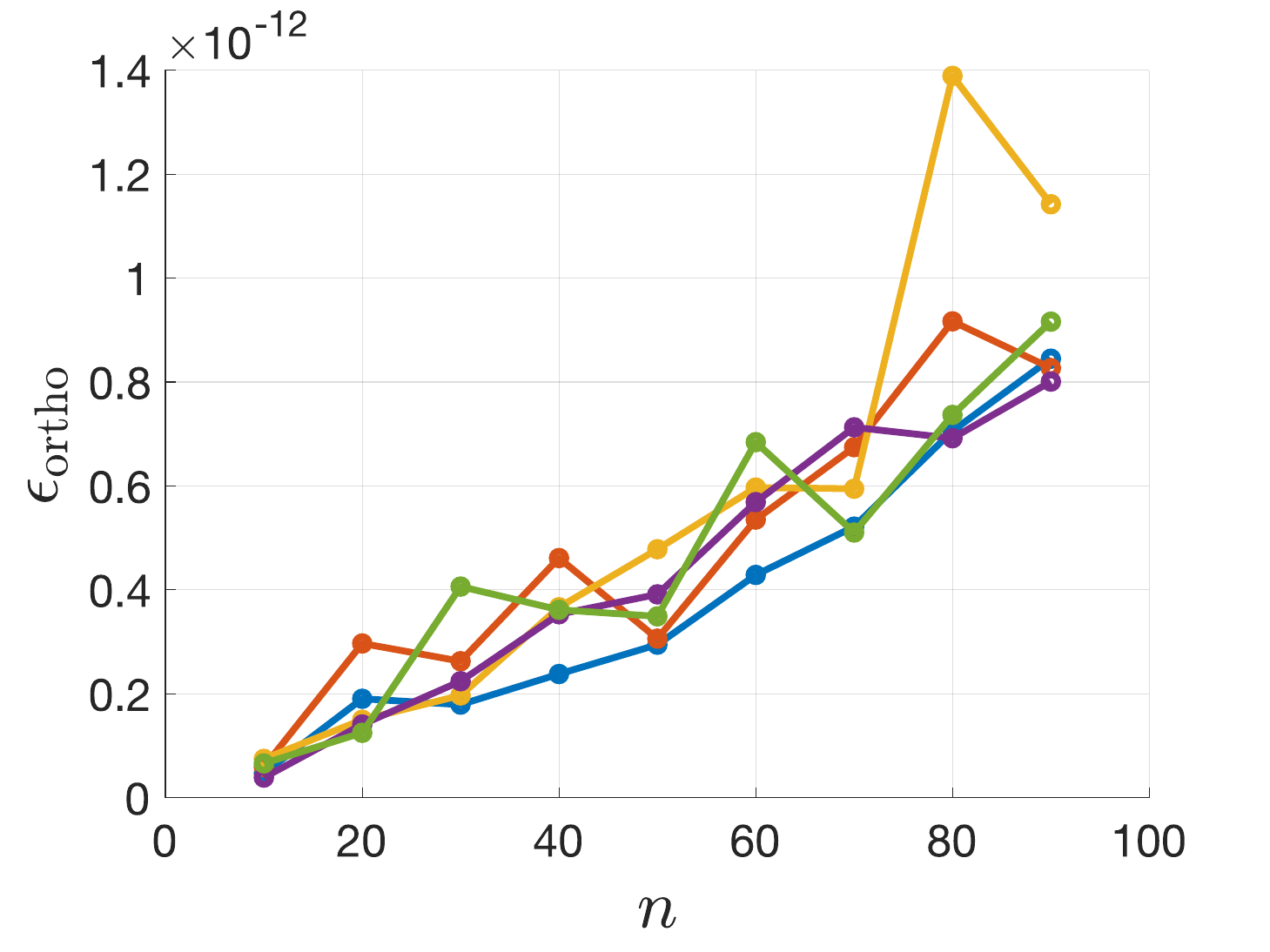}
    \caption{Gram-Schmidt method}
     \label{fig:compar1}
  \end{minipage}
  \hfill
  \begin{minipage}[b]{0.3\textwidth}
    \includegraphics[width=\textwidth]{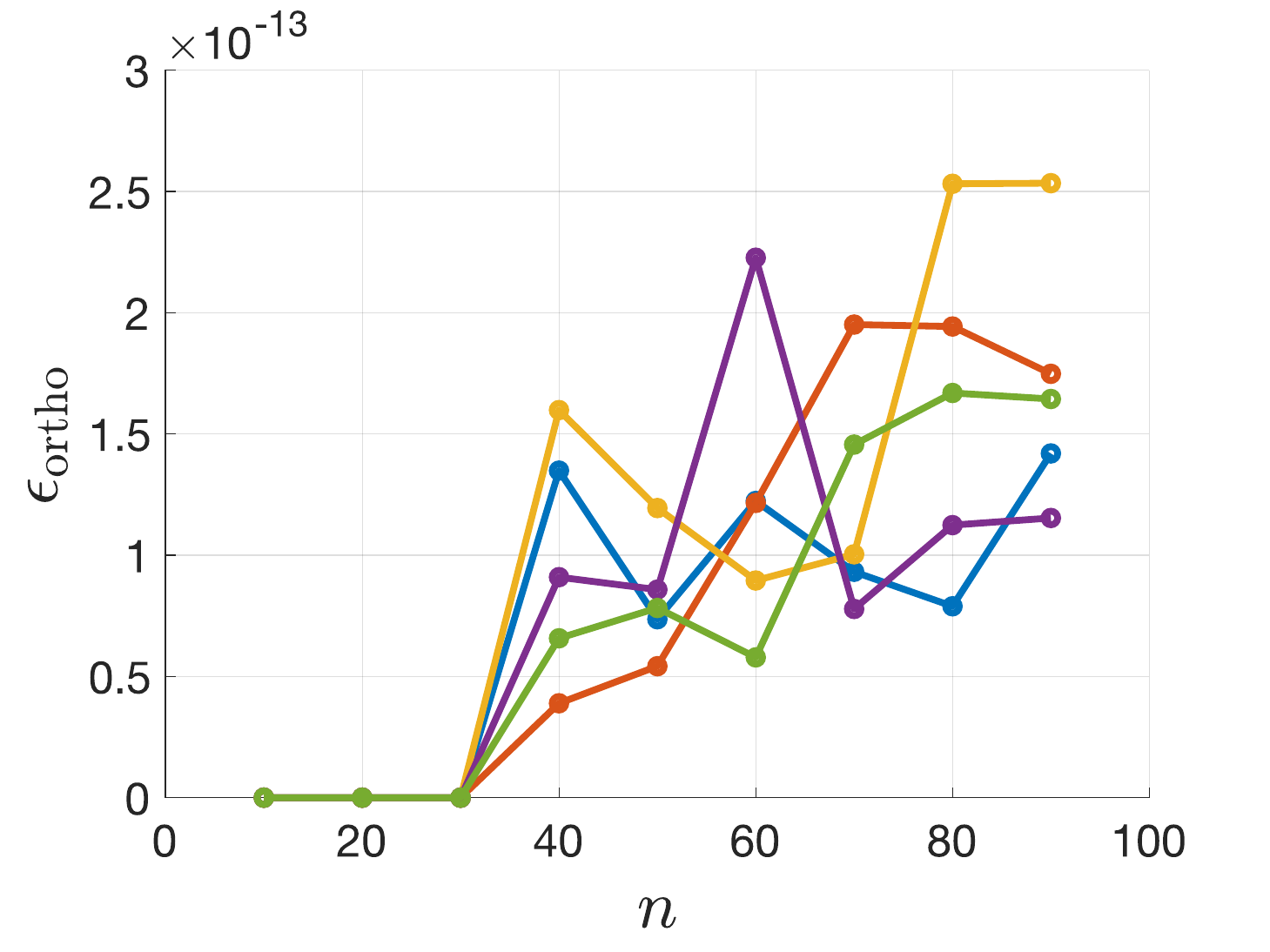}
    \caption{Householder reflections}
     \label{fig:compar2}
  \end{minipage}
  \hfill
  \begin{minipage}[b]{0.3\textwidth}
    \includegraphics[width=\textwidth]{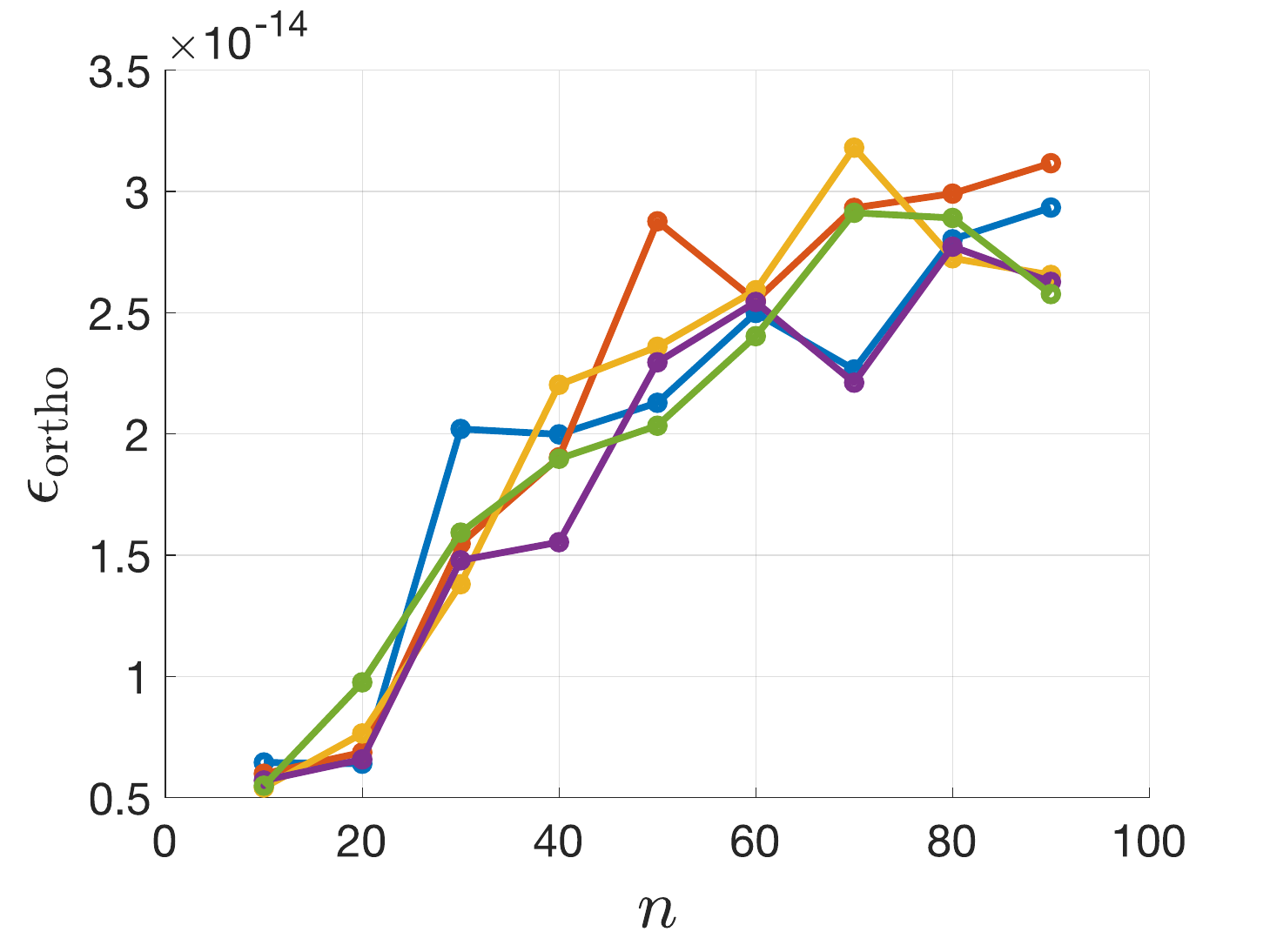}
    \caption{QR decomposition in MATLAB}
      \label{fig:compar3}
  \end{minipage}
\end{figure}

We see that the error increases with an increased number of images. The sequential order of the images generates slight deviations, but the increasing trend remains similar. We can furthermore observe that the numerical errors of the Gram-Schmidt method $(\mathtt{\sim} 10^{-12})$ are $1$ to $2$ magnitudes larger that the orthonormalization errors of the Householder reflections $(\mathtt{\sim} 10^{-13})$ and the QR decomposition in MATLAB $(\mathtt{\sim} 10^{-14})$. So, we choose the QR decomposition for further comparisons.

\subsection{Comparison of QR decomposition with \autoref{eq:gram_matrix_eq}}\label{sec:simulations_QR}
Our goal now is to compare the ``best orthonormalization procedure'', namely the QR decomposition, with the reconstruction via \autoref{eq:gram_matrix_eq}. This comparison is done via the backtransformation of the test data. On the one side, we will compare the two methods in terms of computational efficiency. On the other hand, we will compare the methods in terms of the reconstructed images.

\subsubsection{Computational efficiency}
Here we compare the computational efficiency of \autoref{eq:gram_matrix_eq} with the computational efficiency of the native QR decomposition. In our case, the reconstruction via \autoref{eq:gram_matrix_eq} has clear advantages over the reconstruction via the QR decomposition. For the experiments, a 2,4 GHz 8-Core Intel Core i9 processor is used.

\begin{figure}[H]
  \centering
  \begin{minipage}[b]{0.6\textwidth}
    \includegraphics[width=\textwidth]{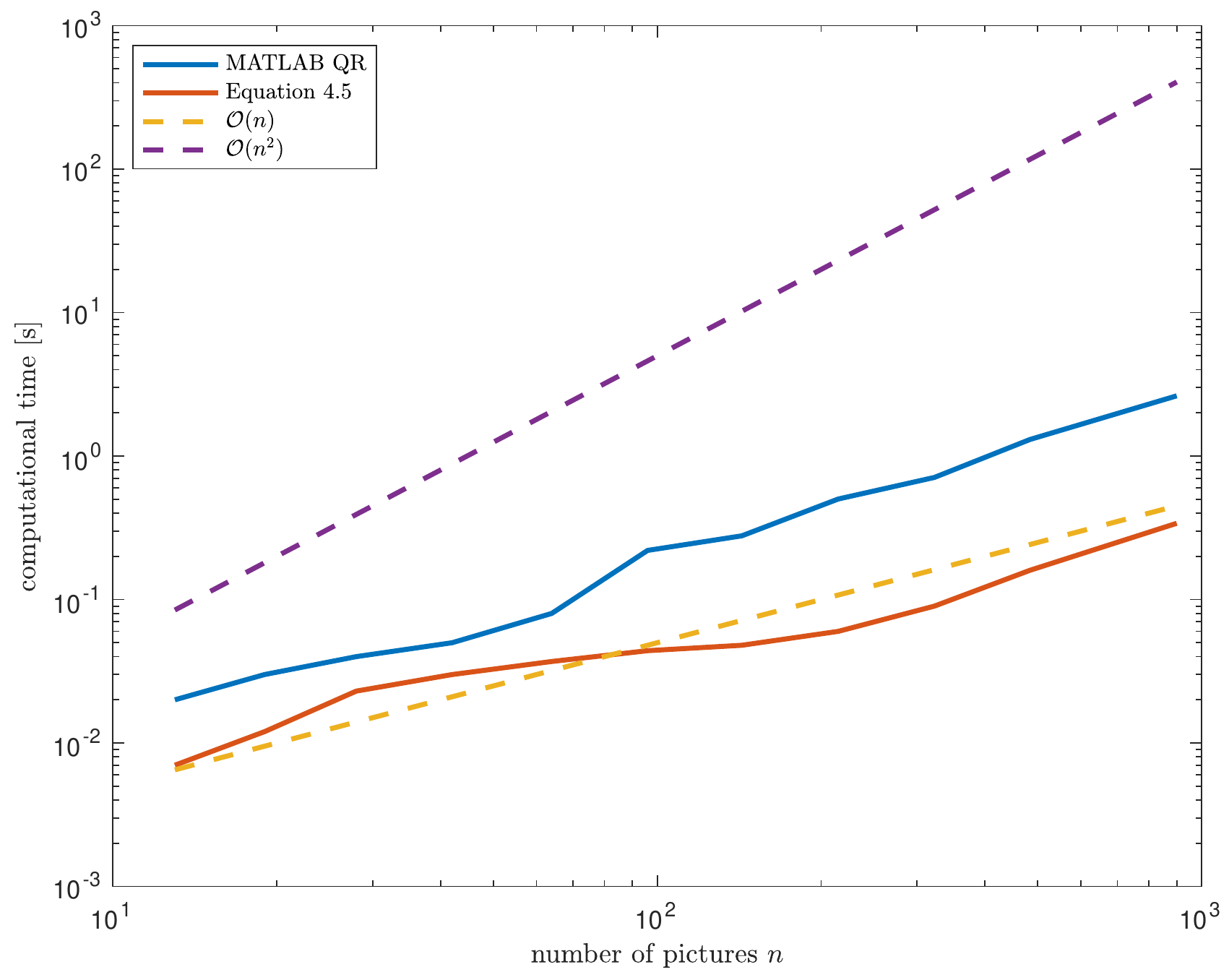}
    \caption{Time comparison}
  \end{minipage}
  
\end{figure}

\subsubsection{Visual observations}
This section observes the reconstructed images of the methods visually. Applying the method on the Radon transforms of the test images, we can compare the original test image with the output of our algorithm. Additionally, we compare the projected Radon transform $Q_n y$ to the Radon transform of the validation data set.
\paragraph{Sunflower data set.} 

We use $n=726$ training images ($150\times 150$ pixels each) of the sunflower data set. Seven additional images, which are not part of the training images are used as test images. These test images contain 4 images with typical motives of sunflowers, where a good approximation on base of the training data is expected and further 3 images with atypical content. On each test image, the reconstruction procedure is applied individually. The results can be seen in Figs.~\ref{fig:test_good} and \ref{fig:test_bad}.\\

We observe a better similarity of the pictures in Fig.~\ref{fig:test_good}, since due to the similarity of sunflowers, the Radon transformations of sunflower motives can be assumed to be closer to the finite dimensional subspace spanned by the training data, than other arbitrary motives. 

Furthermore, we could see that the reconstruction via \autoref{eq:gram_matrix_eq} proceeds at a similar level to the reconstruction via the QR decomposition.  

\paragraph{Digits data set.} 

Similar observations are made with a digits data set \cite{Mnist99} with $n=95$ and $n=995$ training images ($28\times 28$ pixels each) and $5$ test images in Figs.~\ref{fig:digits1}--\ref{fig:digits2}. We see that in the case of $n=995$, the reconstruction via \autoref{eq:gram_matrix_eq} works clearly better than the reconstruction via the QR decomposition. The images are much less blurred.


\begin{figure}[H]
  \centering
  \begin{minipage}[b]{1\textwidth}
    \centering \includegraphics[width=0.10\textwidth,height=0.10\textwidth]{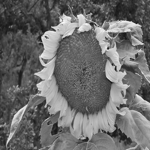}~\includegraphics[width=0.10\textwidth,height=0.10\textwidth]{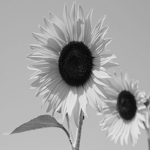}~\includegraphics[width=0.10\textwidth,height=0.10\textwidth]{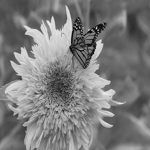}~\includegraphics[width=0.10\textwidth,height=0.10\textwidth]{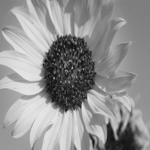}~\includegraphics[width=0.10\textwidth,height=0.10\textwidth]{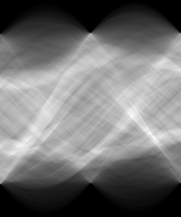}~\includegraphics[width=0.10\textwidth,height=0.10\textwidth]{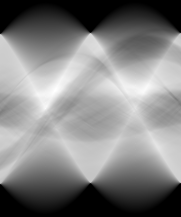}~\includegraphics[width=0.10\textwidth,height=0.10\textwidth]{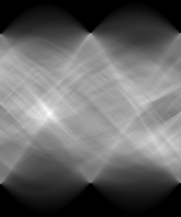}~\includegraphics[width=0.10\textwidth,height=0.10\textwidth]{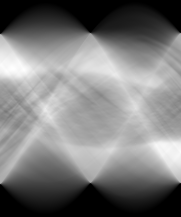}
    \\ (a) Validation data set \\
  \end{minipage}
  \hfill
  \begin{minipage}[b]{1\textwidth}
    \centering \includegraphics[width=0.10\textwidth,height=0.10\textwidth]{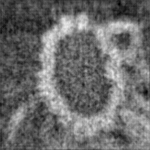}~\includegraphics[width=0.10\textwidth,height=0.10\textwidth]{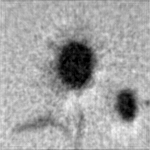}~\includegraphics[width=0.10\textwidth,height=0.10\textwidth]{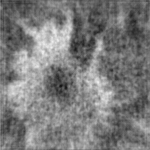}~\includegraphics[width=0.10\textwidth,height=0.10\textwidth]{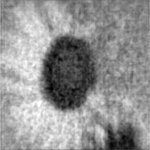}~\includegraphics[width=0.10\textwidth,height=0.10\textwidth]{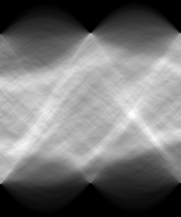}~\includegraphics[width=0.10\textwidth,height=0.10\textwidth]{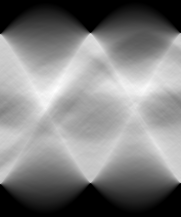}~\includegraphics[width=0.10\textwidth,height=0.10\textwidth]{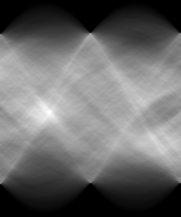}~\includegraphics[width=0.10\textwidth,height=0.10\textwidth]{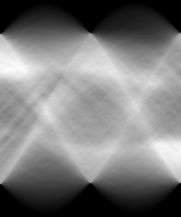}
     \\ (b) QR decomposition \\
  \end{minipage}
  \hfill
  \begin{minipage}[b]{1\textwidth}
    \centering \includegraphics[width=0.10\textwidth,height=0.10\textwidth]{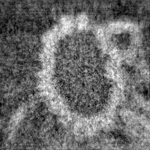}~\includegraphics[width=0.10\textwidth,height=0.10\textwidth]{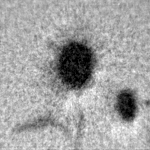}~\includegraphics[width=0.10\textwidth,height=0.10\textwidth]{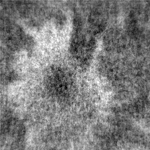}~\includegraphics[width=0.10\textwidth,height=0.10\textwidth]{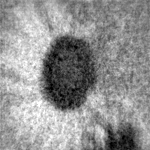}~\includegraphics[width=0.10\textwidth,height=0.10\textwidth]{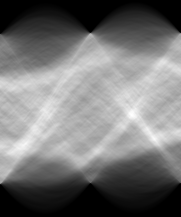}~\includegraphics[width=0.10\textwidth,height=0.10\textwidth]{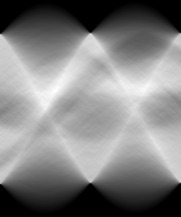}~\includegraphics[width=0.10\textwidth,height=0.10\textwidth]{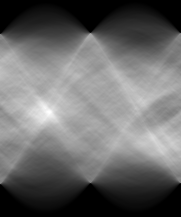}~\includegraphics[width=0.10\textwidth,height=0.10\textwidth]{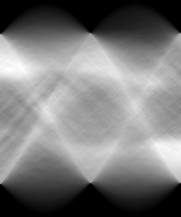}
    \\ (c) Reconstruction with \autoref{eq:gram_matrix_eq} \\

  \end{minipage}
  
  \caption{Reconstructed images via the different methods in comparison to the validation data set for sunflower pictures.}
    \label{fig:test_good}
  
\end{figure}

\begin{figure}[H]
  \centering
  \begin{minipage}[b]{1\textwidth}
    \centering \includegraphics[width=0.10\textwidth,height=0.10\textwidth]{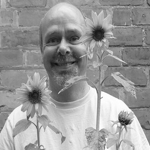}~\includegraphics[width=0.10\textwidth,height=0.10\textwidth]{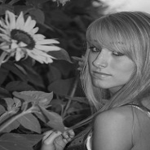}~\includegraphics[width=0.10\textwidth,height=0.10\textwidth]{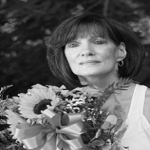}~\includegraphics[width=0.10\textwidth,height=0.10\textwidth]{images/sunflowers/sunflowers_radon1.png}~\includegraphics[width=0.10\textwidth,height=0.10\textwidth]{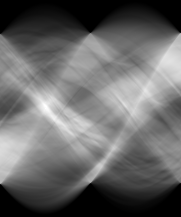}~\includegraphics[width=0.10\textwidth,height=0.10\textwidth]{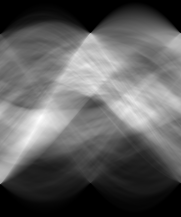}
    \\ (a) Validation data set \\
  \end{minipage}
  \hfill
  \begin{minipage}[b]{1\textwidth}
    \centering \includegraphics[width=0.10\textwidth,height=0.10\textwidth]{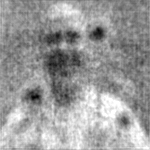}~\includegraphics[width=0.10\textwidth,height=0.10\textwidth]{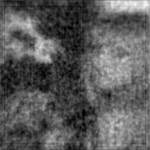}~\includegraphics[width=0.10\textwidth,height=0.10\textwidth]{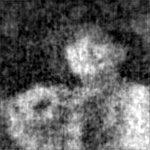}~\includegraphics[width=0.10\textwidth,height=0.10\textwidth]{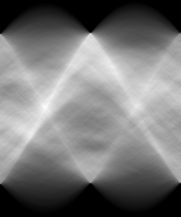}~\includegraphics[width=0.10\textwidth,height=0.10\textwidth]{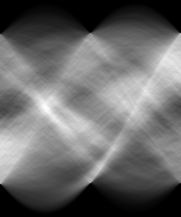}~\includegraphics[width=0.10\textwidth,height=0.10\textwidth]{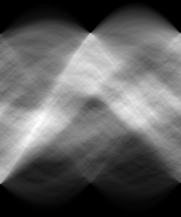}
     \\ (b) QR decomposition \\
  \end{minipage}
  \hfill
  \begin{minipage}[b]{1\textwidth}
    \centering \includegraphics[width=0.10\textwidth,height=0.10\textwidth]{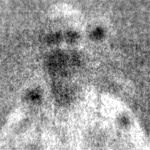}~\includegraphics[width=0.10\textwidth,height=0.10\textwidth]{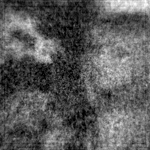}~\includegraphics[width=0.10\textwidth,height=0.10\textwidth]{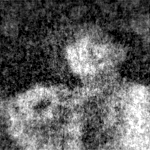}~\includegraphics[width=0.10\textwidth,height=0.10\textwidth]{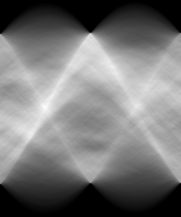}~\includegraphics[width=0.10\textwidth,height=0.10\textwidth]{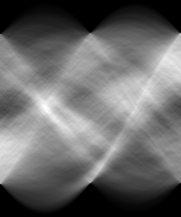}~\includegraphics[width=0.10\textwidth,height=0.10\textwidth]{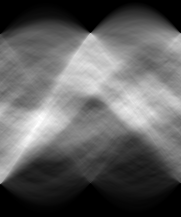}
    \\ (c) Reconstruction with \autoref{eq:gram_matrix_eq} \\

  \end{minipage}
  
  \caption{Reconstructed  images  via  the  different  methods  in  comparison  to  the validation data set for untypical pictures of the data set, which contain a person and where the sunflower motive is only incidental.}
    \label{fig:test_bad}
  
\end{figure}


\begin{figure}[H]
  \centering
  \begin{minipage}[b]{1\textwidth}
    \centering \includegraphics[width=0.10\textwidth,height=0.10\textwidth]{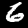}~\includegraphics[width=0.10\textwidth,height=0.10\textwidth]{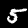}~\includegraphics[width=0.10\textwidth,height=0.10\textwidth]{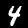}~\includegraphics[width=0.10\textwidth,height=0.10\textwidth]{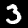}~\includegraphics[width=0.10\textwidth,height=0.10\textwidth]{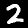}~\includegraphics[width=0.10\textwidth,height=0.10\textwidth]{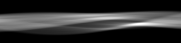}~\includegraphics[width=0.10\textwidth,height=0.10\textwidth]{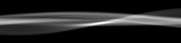}~\includegraphics[width=0.10\textwidth,height=0.10\textwidth]{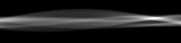}~\includegraphics[width=0.10\textwidth,height=0.10\textwidth]{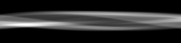}~\includegraphics[width=0.10\textwidth,height=0.10\textwidth]{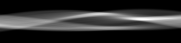}
    \\ (a) Validation data set \\
  \end{minipage}
  \hfill
  \begin{minipage}[b]{1\textwidth}
    \centering \includegraphics[width=0.10\textwidth,height=0.10\textwidth]{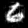}~\includegraphics[width=0.10\textwidth,height=0.10\textwidth]{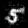}~\includegraphics[width=0.10\textwidth,height=0.10\textwidth]{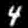}~\includegraphics[width=0.10\textwidth,height=0.10\textwidth]{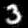}~\includegraphics[width=0.10\textwidth,height=0.10\textwidth]{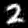}~\includegraphics[width=0.10\textwidth,height=0.10\textwidth]{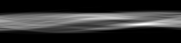}~\includegraphics[width=0.10\textwidth,height=0.10\textwidth]{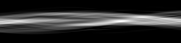}~\includegraphics[width=0.10\textwidth,height=0.10\textwidth]{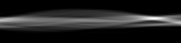}~\includegraphics[width=0.10\textwidth,height=0.10\textwidth]{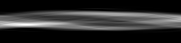}~\includegraphics[width=0.10\textwidth,height=0.10\textwidth]{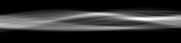}
     \\ (b) QR decomposition \\
  \end{minipage}
  \hfill
  \begin{minipage}[b]{1\textwidth}
    \centering \includegraphics[width=0.10\textwidth,height=0.10\textwidth]{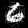}~\includegraphics[width=0.10\textwidth,height=0.10\textwidth]{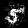}~\includegraphics[width=0.10\textwidth,height=0.10\textwidth]{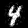}~\includegraphics[width=0.10\textwidth,height=0.10\textwidth]{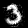}~\includegraphics[width=0.10\textwidth,height=0.10\textwidth]{images/numbers_dataset_100/qr_numbers_dataset5.png}~\includegraphics[width=0.10\textwidth,height=0.10\textwidth]{images/numbers_dataset_100/qr_numbers_dataset_radon1.png}~\includegraphics[width=0.10\textwidth,height=0.10\textwidth]{images/numbers_dataset_100/qr_numbers_dataset_radon2.png}~\includegraphics[width=0.10\textwidth,height=0.10\textwidth]{images/numbers_dataset_100/qr_numbers_dataset_radon3.png}~\includegraphics[width=0.10\textwidth,height=0.10\textwidth]{images/numbers_dataset_100/qr_numbers_dataset_radon4.png}~\includegraphics[width=0.10\textwidth,height=0.10\textwidth]{images/numbers_dataset_100/qr_numbers_dataset_radon5.png}
    \\ (c) Reconstruction with \autoref{eq:gram_matrix_eq} \\

  \end{minipage}
  
  \caption{Number of training images: 95.}
    \label{fig:digits1}
  
\end{figure}



\begin{figure}[H]
  \centering
  \begin{minipage}[b]{1\textwidth}
    \centering \includegraphics[width=0.10\textwidth,height=0.10\textwidth]{images/numbers_dataset_100/numbers_dataset1.png}~\includegraphics[width=0.10\textwidth,height=0.10\textwidth]{images/numbers_dataset_100/numbers_dataset2.png}~\includegraphics[width=0.10\textwidth,height=0.10\textwidth]{images/numbers_dataset_100/numbers_dataset3.png}~\includegraphics[width=0.10\textwidth,height=0.10\textwidth]{images/numbers_dataset_100/numbers_dataset4.png}~\includegraphics[width=0.10\textwidth,height=0.10\textwidth]{images/numbers_dataset_100/numbers_dataset5.png}~\includegraphics[width=0.10\textwidth,height=0.10\textwidth]{images/numbers_dataset_100/numbers_radon1.png}~\includegraphics[width=0.10\textwidth,height=0.10\textwidth]{images/numbers_dataset_100/numbers_radon2.png}~\includegraphics[width=0.10\textwidth,height=0.10\textwidth]{images/numbers_dataset_100/numbers_radon3.png}~\includegraphics[width=0.10\textwidth,height=0.10\textwidth]{images/numbers_dataset_100/numbers_radon4.png}~\includegraphics[width=0.10\textwidth,height=0.10\textwidth]{images/numbers_dataset_100/numbers_radon5.png}
     \\ (a) Validation data set \\
  \end{minipage}
  \hfill
  \begin{minipage}[b]{1\textwidth}
    \centering \includegraphics[width=0.10\textwidth,height=0.10\textwidth]{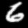}~\includegraphics[width=0.10\textwidth,height=0.10\textwidth]{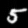}~\includegraphics[width=0.10\textwidth,height=0.10\textwidth]{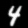}~\includegraphics[width=0.10\textwidth,height=0.10\textwidth]{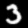}~\includegraphics[width=0.10\textwidth,height=0.10\textwidth]{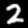}~\includegraphics[width=0.10\textwidth,height=0.10\textwidth]{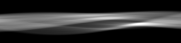}~\includegraphics[width=0.10\textwidth,height=0.10\textwidth]{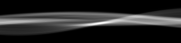}~\includegraphics[width=0.10\textwidth,height=0.10\textwidth]{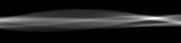}~\includegraphics[width=0.10\textwidth,height=0.10\textwidth]{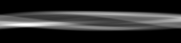}~\includegraphics[width=0.10\textwidth,height=0.10\textwidth]{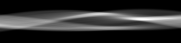}
     \\ (b) QR decomposition \\
  \end{minipage}
  \hfill
  \begin{minipage}[b]{1\textwidth}
    \centering  \includegraphics[width=0.10\textwidth,height=0.10\textwidth]{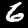}~\includegraphics[width=0.10\textwidth,height=0.10\textwidth]{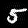}~\includegraphics[width=0.10\textwidth,height=0.10\textwidth]{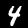}~\includegraphics[width=0.10\textwidth,height=0.10\textwidth]{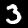}~\includegraphics[width=0.10\textwidth,height=0.10\textwidth]{images/numbers_dataset_1000/qr_numbers_dataset5.png}~\includegraphics[width=0.10\textwidth,height=0.10\textwidth]{images/numbers_dataset_100/qr_numbers_dataset_radon1.png}~\includegraphics[width=0.10\textwidth,height=0.10\textwidth]{images/numbers_dataset_1000/qr_numbers_dataset_radon2.png}~\includegraphics[width=0.10\textwidth,height=0.10\textwidth]{images/numbers_dataset_100/qr_numbers_dataset_radon3.png}~\includegraphics[width=0.10\textwidth,height=0.10\textwidth]{images/numbers_dataset_1000/qr_numbers_dataset_radon4.png}~\includegraphics[width=0.10\textwidth,height=0.10\textwidth]{images/numbers_dataset_1000/qr_numbers_dataset_radon5.png}
    \\ (c) Reconstruction with \autoref{eq:gram_matrix_eq} \\

  \end{minipage}
  
  \caption{Number of training images: 995.}
    \label{fig:digits2}
  
\end{figure}

\section{Conclusions}
We have adapted the projection method applied in frames theory to a data driven reconstruction algorithm for solving inverse problems. We have shown some numerical experiments comparing the reconstruction formula \autoref{eq:gram_matrix_eq} with the one in \autoref{eq:udn2}, proposed by the authors in \cite{AspKorSch20}, which is based on orthonormalization procedures. 
Numerical results based on the frame method are promising, shown that, with a big amount of training pairs, reconstructions are better than those provided by \autoref{eq:udn2}, see, for example, Figure \ref{fig:digits2} and with a lower cost in terms of computational time. Studies on convergence results and rates will be the focus of future works.    
\section*{Acknowledgments}
LF and OS are supported by the FWF via the projects I3661-N27 (Novel Error Measures and Source Conditions of Regularization Methods for Inverse Problems). OS is also supported by FWF via SFB F68, project F6807-N36 (Tomography with Uncertainties). YK acknowledges the support of the EPSRC (Fellowship EP/V003615/1), the Cantab Capital Institute for the Mathematics of Information and the National Physical Laboratory.

\section*{References}
\renewcommand{\i}{\ii}
\printbibliography[heading=none]

@article {AdcHuy19,
    AUTHOR = {Adcock, Ben and Huybrechs, Daan},
     TITLE = {Frames and numerical approximation},
   JOURNAL = {SIAM Rev.},
  FJOURNAL = {SIAM Review},
    VOLUME = {61},
      YEAR = {2019},
    NUMBER = {3},
     PAGES = {443--473},
      ISSN = {0036-1445},
   MRCLASS = {42C15 (42C30 65T40)},
  MRNUMBER = {3989238},
MRREVIEWER = {Fritz Keinert},
       DOI = {10.1137/17M1114697},
       URL = {https://doi.org/10.1137/17M1114697},
}

@article{ArrMaaOktScho19, 
author={Arridge, Simon and Maass, Peter and Öktem, Ozan and Schönlieb, Carola-Bibiane}, 
title={Solving inverse problems using data-driven models}, volume={28}, 
DOI={10.1017/S0962492919000059}, 
journal={Acta Numerica}, 
publisher={Cambridge University Press}, 
year={2019}, 
pages={1–174}}

@article {AspKorSch20,
    AUTHOR = {Aspri, Andrea and Korolev, Yury and Scherzer, Otmar},
     TITLE = {Data driven regularization by projection},
   JOURNAL = {Inverse Problems},
  FJOURNAL = {Inverse Problems. An International Journal on the Theory and
              Practice of Inverse Problems, Inverse Methods and Computerized
              Inversion of Data},
    VOLUME = {36},
      YEAR = {2020},
    NUMBER = {12},
     PAGES = {125009, 35},
      ISSN = {0266-5611},
   MRCLASS = {65J22 (47A52)},
  MRNUMBER = {4186179},
       DOI = {10.1088/1361-6420/abb61b},
       URL = {https://doi.org/10.1088/1361-6420/abb61b},
}

@article {Chr96,
    AUTHOR = {Christensen, Ole},
     TITLE = {Frames containing a {R}iesz basis and approximation of the
              frame coefficients using finite-dimensional methods},
   JOURNAL = {J. Math. Anal. Appl.},
  FJOURNAL = {Journal of Mathematical Analysis and Applications},
    VOLUME = {199},
      YEAR = {1996},
    NUMBER = {1},
     PAGES = {256--270},
      ISSN = {0022-247X},
   MRCLASS = {46B15 (46C05)},
  MRNUMBER = {1381391},
MRREVIEWER = {Jay Epperson},
       DOI = {10.1006/jmaa.1996.0140},
       URL = {https://doi.org/10.1006/jmaa.1996.0140},
}

@book {Chr16,
    AUTHOR = {Christensen, Ole},
     TITLE = {An introduction to frames and {R}iesz bases},
    SERIES = {Applied and Numerical Harmonic Analysis},
   EDITION = {Second},
 PUBLISHER = {Birkh\"{a}user/Springer},
      YEAR = {2016},
     PAGES = {xxv+704},
      ISBN = {978-3-319-25611-5; 978-3-319-25613-9},
   MRCLASS = {42-02 (42C15 42C40 46B15 46C05)},
  MRNUMBER = {3495345},
MRREVIEWER = {Marcin M. Bownik},
       DOI = {10.1007/978-3-319-25613-9},
       URL = {https://doi.org/10.1007/978-3-319-25613-9},
}

@book {EngHanNeu96,
    AUTHOR = {Engl, Heinz W. and Hanke, Martin and Neubauer, Andreas},
     TITLE = {Regularization of inverse problems},
    SERIES = {Mathematics and its Applications},
    VOLUME = {375},
 PUBLISHER = {Kluwer Academic Publishers Group, Dordrecht},
      YEAR = {1996},
     PAGES = {viii+321},
      ISBN = {0-7923-4157-0},
   MRCLASS = {65J20 (35R30 47A50 47H17)},
  MRNUMBER = {1408680},
MRREVIEWER = {Ulrich Tautenhahn},
}

@book {Golub96,
    AUTHOR = {Golub, Gene H. and Van Loan, Charles F.},
     TITLE = {Matrix computations},
    SERIES = {Johns Hopkins Studies in the Mathematical Sciences},
   EDITION = {Third},
 PUBLISHER = {Johns Hopkins University Press, Baltimore, MD},
      YEAR = {1996},
     PAGES = {xxx+698},
      ISBN = {0-8018-5413-X; 0-8018-5414-8},
   MRCLASS = {65-02 (65Fxx)},
  MRNUMBER = {1417720},
}

@incollection {Kuc13,
    AUTHOR = {Kuchment, Peter},
     TITLE = {Mathematics of hybrid imaging: a brief review},
 BOOKTITLE = {The mathematical legacy of {L}eon {E}hrenpreis},
    SERIES = {Springer Proc. Math.},
    VOLUME = {16},
     PAGES = {183--208},
 PUBLISHER = {Springer, Milan},
      YEAR = {2012},
   MRCLASS = {44A12 (35R30 92C55)},
  MRNUMBER = {3289684},
MRREVIEWER = {Denis N. Sidorov},
       DOI = {10.1007/978-88-470-1947-8_12},
       URL = {https://doi.org/10.1007/978-88-470-1947-8_12},
}

@incollection {Isa90,
    AUTHOR = {Isakov, Victor},
     TITLE = {Some inverse problems for elliptic and parabolic equations},
 BOOKTITLE = {Inverse problems in partial differential equations ({A}rcata,
              {CA}, 1989)},
     PAGES = {203--214},
 PUBLISHER = {SIAM, Philadelphia, PA},
      YEAR = {1990},
   MRCLASS = {35R30 (35J25 35K20)},
  MRNUMBER = {1046438},
MRREVIEWER = {Sybille Handrock-Meyer},
}

@proceedings{Nas76,
	Address = {New York},
	Booktitle = {Proceedings of an Advanced Seminar sponsored by the Mathematics Research Center at the University of Wisconsin, Madison, Wis., October 8-10, 1973},
	Editor = {Nashed, M.Z.},
	Pages = {xiv+1054},
	Publisher = {Academic Press [Harcourt Brace Jovanovich Publishers]},
	Title = {Generalized inverses and applications},
	Year = {1976}}

@book {Nat01,
    AUTHOR = {Natterer, F.},
     TITLE = {The mathematics of computerized tomography},
    SERIES = {Classics in Applied Mathematics},
    VOLUME = {32},
      NOTE = {Reprint of the 1986 original},
 PUBLISHER = {Society for Industrial and Applied Mathematics (SIAM),
              Philadelphia, PA},
      YEAR = {2001},
     PAGES = {xviii+222},
      ISBN = {0-89871-493-1},
   MRCLASS = {00A69 (44A12 65R10 68U99 92C55)},
  MRNUMBER = {1847845},
MRREVIEWER = {Fritz Keinert},
       DOI = {10.1137/1.9780898719284},
       URL = {https://doi.org/10.1137/1.9780898719284},
}

@book{NatWue01,
author = {Natterer, F. and W\"ubbeling, F.},
DOI = {10.1137/1.9780898718324},
publisher = {Society for Industrial and Applied Mathematics},
title = {Mathematical Methods in Image Reconstruction},
year = {2001}}

@book {SchGraGroHalLen09,
    AUTHOR = {Scherzer, Otmar and Grasmair, Markus and Grossauer, Harald and
              Haltmeier, Markus and Lenzen, Frank},
     TITLE = {Variational methods in imaging},
    SERIES = {Applied Mathematical Sciences},
    VOLUME = {167},
 PUBLISHER = {Springer, New York},
      YEAR = {2009},
     PAGES = {xiv+320},
      ISBN = {978-0-387-30931-6},
   MRCLASS = {49-02 (49J10 68U10 94A08)},
  MRNUMBER = {2455620},
MRREVIEWER = {Bogdan G. Nita},
}

@article {Sei80,
    AUTHOR = {Seidman, T. I.},
     TITLE = {Nonconvergence results for the application of least-squares
              estimation to ill-posed problems},
   JOURNAL = {J. Optim. Theory Appl.},
  FJOURNAL = {Journal of Optimization Theory and Applications},
    VOLUME = {30},
      YEAR = {1980},
    NUMBER = {4},
     PAGES = {535--547},
      ISSN = {0022-3239},
   MRCLASS = {65J10},
  MRNUMBER = {572154},
MRREVIEWER = {Thomas S. Angell},
       DOI = {10.1007/BF01686719},
       URL = {https://doi.org/10.1007/BF01686719},
}

@article {Sym09,
    AUTHOR = {Symes, W. W.},
     TITLE = {The seismic reflection inverse problem},
   JOURNAL = {Inverse Problems},
  FJOURNAL = {Inverse Problems. An International Journal on the Theory and
              Practice of Inverse Problems, Inverse Methods and Computerized
              Inversion of Data},
    VOLUME = {25},
      YEAR = {2009},
    NUMBER = {12},
     PAGES = {123008, 39},
      ISSN = {0266-5611},
   MRCLASS = {86A15 (35S05 86A22)},
  MRNUMBER = {2565574},
MRREVIEWER = {Sven Ivansson},
       DOI = {10.1088/0266-5611/25/12/123008},
       URL = {https://doi.org/10.1088/0266-5611/25/12/123008},
}

@ARTICLE{Mnist99,  
author={Y. {Lecun} and L. {Bottou} and Y. {Bengio} and P. {Haffner}},  
journal={Proceedings of the IEEE},   
title={Gradient-based learning applied to document recognition},   
year={1998},  
volume={86},  
number={11},  
pages={2278-2324},  
doi={10.1109/5.726791}}

\end{document}